\documentclass[12pt]{article}
\usepackage{amsmath,amssymb,  amsthm, stmaryrd, MnSymbol, hyperref}
\usepackage{amsfonts}
\usepackage{graphicx}
\usepackage{color}
\usepackage{enumerate}
\usepackage{accents}

\newcommand{\utilde}[1]{\underaccent{\tilde}{#1}}

\newtheorem{theo}{Theorem}[section]

\newtheorem{coro}[theo]{Corollary}

\newtheorem{lemma}[theo]{Lemma}
\newtheorem{re}[theo]{Remark}

\newtheorem{con}[theo]{Convention}
\theoremstyle{definition}
\newtheorem{definition}[theo]{Definition}

\begin{document}
	\title{\bf A Set-Theoretic Translation of Modal Logic via Forcing}
	
\author{
	Somayeh Chopoghloo\thanks{School of Mathematics, Institute for Research in Fundamental Sciences (IPM), P.O. Box 19395-5746, Tehran, Iran. Email: \texttt{s.chopoghloo@gmail.com}.}
	\and
	Mohammad Golshani\footnotemark[1] \thanks{Email: \texttt{golshani.m@gmail.com}.}
}	
	
\date{}
\maketitle
\begin{abstract}
We develop a set-theoretic translation of a normal modal extension $T_m$ of a recursively axiomatizable first-order theory $T$. We first pass to the Henkin expansion of the underlying language by adding witness constants and work with the sentence algebra of this expansion. The translation is constructed using the corresponding Lindenbaum-Tarski algebra, the Stone space of its ultrafilters, and quotient forcing by a $\sigma$-ideal of Borel sets; in particular, the meager ideal yields Cohen forcing and the null ideal yields the random forcing. In a Boolean-valued universe, we interpret the modal operator $\Box$ by membership of the Boolean value of the translated formula in a suitable filter name. We prove that this translation is sound and complete: a modal formula is provable in $T_m$ if and only if its set-theoretic translation is forced in every associated interpretation. We then build Kripke frames and models from generic extensions and show that the forcing interpretation of $\Box$ agrees with quantification over the corresponding accessibility relation. As a consequence, we obtain completeness with respect to the resulting Kripke models.
\end{abstract}

\section{Introduction}\label{Sec:Intro}

Modal logic \cite{BRV01,Chagrov1997} provides a formal framework for reasoning about notions such as necessity, possibility, and accessibility. Its first-order extension enriches the propositional modal language by incorporating variables, function, relation, and constant symbols, equality, and quantifiers, thereby enabling reasoning about individuals across possible worlds. As a result, first-order modal logic has become a fundamental formalism in logic, philosophy, linguistics, and computer science. Its semantics is typically given in terms of {\em first-order Kripke models}, which may be interpreted over constant, varying, or increasing domains of quantification. These different domain conditions give rise to distinct logical systems with different expressive power and metatheoretical properties. Comprehensive treatments of the syntax, semantics, completeness and metatheory of first-order modal logic can be found in Fitting and Mendelsohn \cite{Fitting1998}, Garson \cite{Garson2023} and the survey by Braüner and Ghilardi \cite{Brau}.

In this paper, we study a set-theoretic interpretation of modal logic based on Boolean algebras, Stone spaces, and forcing. Our starting point is a consistent recursively
axiomatizable first-order theory $T$ in a language $\mathcal{L}$, which we assume  is Henkinian.
 We then consider the normal modal extension $T_m$ in the language
$\mathcal{L}_m=\mathcal{L}\cup\{\Box,\lozenge\}$. The goal is to associate modal formulas with set-theoretic formulas in such a way that the modal principles of $T_m^{H}$ are reflected in the corresponding forcing semantics.

The construction is carried out using the Lindenbaum-Tarski algebras
$\mathsf{B}_{T}$ and $\mathsf{B}_{T_m}$.  We fix
$\mathsf{B}=\mathsf{B}_{T}$ and consider its Stone space
$\operatorname{St}(\mathsf{B})$. By assuming $T$ extend's Robinson's arithmetic $Q$, the Boolean algebra $\mathsf{B}$ is atomless.  If $I$ is a proper $\sigma$-ideal on the
Borel subsets of $\operatorname{St}(\mathsf{B})$, we form the quotient
forcing notion
\[
\mathsf{P}_I=(\mathcal{B}\bigl(\operatorname{St}(\mathsf{B})\bigr)/I)^+,
\]
which is obtained from the Boolean algebra $\mathcal{B}\bigl(\operatorname{St}(\mathsf{B})\bigr)$ by removing its zero element.
The meager ideal gives Cohen forcing, while the null ideal gives
random forcing.

We then define recursively a set-theoretic translation
$\utilde{\nu}(\varphi)$ for each sentence $\varphi\in\mathcal{L}_m$.
This translation is interpreted in the Boolean-valued universe
associated with
$
\mathsf{C}_I=\mathsf{RO}(\mathsf{P}_I).
$
The Stone space of $\mathsf{C}_I$ is represented by a
$\mathsf{C}_I$-name $\operatorname{\utilde{St}}(\mathsf{C}_I)$, and
the generic filter $\dot{G}_{\mathsf{C}_I}$ is forced to be an
ultrafilter on $\mathsf{C}_I$. A $\mathsf{C}_I$-name $\utilde{J}$
interprets the constant symbols of $\mathcal{L}$ in this
Boolean-valued setting. The propositional connectives and first-order
quantifiers are translated recursively, while atomic formulas are
represented by membership of their corresponding quotient conditions in
the generic filter. The modal operator is interpreted by
\[
\utilde{\nu}(\Box\varphi)=[\![\utilde{\nu}(\varphi)]\!]\in\utilde{N},
\]
where $\utilde{N}$ is a suitable $\mathsf{C}_I$-name for a filter.

We prove that this translation respects the logical structure of
$T_m$. In particular, it preserves the first-order consequences of
$T$, the modal axioms
\[
\Box\varphi\to\varphi
\qquad\text{and}\qquad
\Box(\varphi\to\psi)\to(\Box\varphi\to\Box\psi),
\]
as well as the necessitation rule. These results yield a soundness
theorem: if $T_m\vdash\varphi$, then the translation of $\varphi$ is
forced in every interpretation determined by the relevant names
$\utilde{N}$ and $\utilde{J}$. Conversely, if the translation of
$\varphi$ is forced in all such interpretations, then
$T_m\vdash\varphi$. Thus the set-theoretic translation is both sound
and complete for the modal theory under consideration.

Finally, we associate Kripke frames and Kripke models with the
forcing construction. The worlds are given by generic extensions
$V[G]$ of the ground model $V$, and the structure at each world is
obtained from the interpreted Stone space together with the
interpretations of the function, relation, and constant symbols of
$\mathcal{L}$. The accessibility relation between generic extensions is
defined by preservation of translated formulas: $V[G]$ is accessible to
$V[H]$ when every translated formula true in $V[G]$ remains true in
$V[H]$. We show that the forcing interpretation of $\Box$ agrees with
the corresponding Kripke semantics over this accessibility relation.

The main consequence is a completeness theorem for the resulting
Kripke models. Namely, a formula $\varphi$ is provable in $T_m$ if and
only if its set-theoretic translation is valid in every Kripke model
obtained from the forcing construction. This establishes a direct
connection between modal provability, Boolean-valued set theory,
quotient forcing, and Kripke semantics. Our results extend the work of Boos \cite{Boos} in the case of Cohen forcing.
\section{Preliminaries}\label{Sec:Pre}

This section reviews the basic concepts and results used throughout the paper. We begin with a brief overview of normal propositional modal logics and their first-order extensions. For comprehensive treatments of the syntax, semantics, completeness, and metatheory of first-order modal logic, we refer the reader to \cite{Fitting1998,Brau}. Our presentation of Stone spaces, regular open algebras, Boolean-valued universes, and generic extensions is based primarily on \cite{Viale2024}; further background can be found in \cite{Jech2003,Bell2005}. Finally, for Henkin expansions and sentence algebras, we refer to \cite{ChangKeisler}.
\subsection{Modal logic}

The {\em language of propositional modal logic} is obtained by extending the language of classical propositional logic with the modal operators $\Box$ (necessity) and $\lozenge$ (possibility). Intuitively, the formula $\Box\varphi$ expresses that ``$\varphi$ holds in every accessible world" while $\lozenge\varphi$ expresses that ``$\varphi$ holds in some accessible world". The two modalities are interdefinable, where $\lozenge\varphi\;:=\;\neg\Box\neg\varphi$.

A {\em normal propositional modal logic} is a set of propositional modal formulas containing all propositional tautologies and the axiom schema \[\Box(\varphi\rightarrow\psi) \rightarrow (\Box\varphi\rightarrow\Box\psi)\] and closed under the following inference rules
\[\frac{\varphi\qquad\varphi\rightarrow\psi}{\psi} \quad(\textit{Modus Ponens})
\qquad\text{and}\qquad\frac{\varphi}{\Box\varphi} \quad(\textit{Necessitation}).\] The smallest normal propositional modal logic is denoted by $\mathsf{K}$. It is sound and complete with respect to the class of all propositional Kripke frames. A {\em propositional Kripke frame} is a pair $\langle W, \triangleleft\rangle$ where $W$ is a nonempty set of worlds and $\triangleleft\subseteq W\times W$ is an accessibility relation on $W$. Many well-known normal modal logics are obtained by adding axioms that correspond to natural first-order properties of the accessibility relation. For example,
\begin{itemize}
	\item
	$\mathsf{T}:=\mathsf{K}+(\Box\varphi\rightarrow\varphi)$ is sound and
	complete for the class of reflexive frames.
	\item
	$\mathsf{S4}:=\mathsf{T}+(\Box\varphi\rightarrow\Box\Box\varphi)$ is
	sound and complete with respect to the class of reflexive and
	transitive frames (or preorders).
\end{itemize}

The {\em first-order modal language} is obtained by extending a first-order language with the modal operators $\Box$ and $\lozenge$. More precisely, if $\mathcal{L}$ is a first-order language with equality, then the corresponding first-order modal language is $\mathcal{L}_m=\mathcal{L}\cup\{\Box,\lozenge\}$,
whose formulas are generated from the atomic formulas of $\mathcal{L}$ using the usual Boolean connectives, quantifiers, and the modal operators.

The semantics of first-order modal logic is given by first-order Kripke models in which each world is equipped with a first-order structure.

\begin{definition}
	A {\em first-order Kripke frame} is a tuple $\mathfrak{F}=\langle W,\triangleleft,D\rangle$ where $\langle W,\triangleleft\rangle$ is a propositional Kripke frame and 	$D=\{D_w\}_{w\in W}$ is a family of nonempty sets indexed by $W$. For each $w\in W$, the set $D_w$ is called the {\em domain} at $w$.
\end{definition}

\begin{definition}
	A {\em first-order Kripke model} is a pair $\mathfrak{M}=\langle\mathfrak{F},I\rangle$ where
	$\mathfrak{F}=\langle W,\triangleleft,D \rangle$ is a first-order Kripke frame and $I=\{I_w\}_{w\in W}$ is an interpretation assigning to each world
	$w\in W$,
		\begin{itemize}
			\item an element of $D_w$ to every constant symbol,
			\item a function $D_w^n\to D_w$ to every $n$-ary function symbol,
			\item a relation $I_w(R)\subseteq D_w^n$ to every $n$-ary relation symbol.
		\end{itemize}
\end{definition}

The \emph{forcing relation} $\Vdash$ is defined recursively between
first-order Kripke models, worlds, and $\mathcal{L}_m$-formulas. In particular,
\[
\mathfrak{M},w\Vdash\Box\varphi
\quad\Longleftrightarrow\quad
\mathfrak{M},w'\Vdash\varphi
\text{ for every }w'\in W\text{ such that }w\triangleleft w'.
\]

\begin{re}
In order to the above definition be well-defined, usually one needs to assume either the domains are constant or increasing.
In our case, we have a coherent family $\langle \bold{t}_{w, w'}:  w\triangleleft w' \rangle$ of maps
$\bold{t}_{w, w'}: D_w \rightarrow D_{w'}$ such that for $ w\triangleleft w'  \triangleleft w'' $
we have $\bold{t}_{w, w''}=\bold{t}_{w', w''} \circ \bold{t}_{w, w'}$.
\end{re}

Let $\mathfrak{M}=\langle\mathfrak{F},I\rangle$ be a first-order Kripke model where $\mathfrak{F}=\langle W,\triangleleft,D\rangle$ and let $\varphi$ be an $\mathcal{L}_m$-sentence. We say that $\varphi$ is {\em true} in $\mathfrak{M}$, written $\mathfrak{M}\Vdash\varphi$, if $\mathfrak{M},w\Vdash\varphi$ for every world $w\in W$. If $\varphi$ is an open formula, we say that $\varphi$ is true in $\mathfrak{M}$ if its universal closure is true in $\mathfrak{M}$.

Let $\mathfrak{F}_0=\langle W,\triangleleft\rangle$ denote the underlying
propositional Kripke frame of $\mathfrak{F}$. An $\mathcal L_m$-formula $\varphi$ is {\em valid} in $\mathfrak F_0$ if $\mathfrak{M}\Vdash\varphi$ for every first-order Kripke model $\mathfrak{M}$ based on $\mathfrak{F}_0$.
More generally, if $\mathcal{C}$ is a class of propositional Kripke frames, then
$\varphi$ is {\em valid} in $\mathcal{C}$ if it is valid in every frame in $\mathcal{C}$.

Let $\Gamma$ be a set of $\mathcal{L}_m$-sentences.
\begin{itemize}
	\item A first-order Kripke model $\mathfrak{M}$ is a
	{\em model of $\Gamma$}, written $\mathfrak{M}\Vdash \Gamma$, if and only if $\mathfrak{M}\Vdash\varphi$ for every $\varphi\in \Gamma$.
	\item Let $\mathcal C$ be a class of propositional Kripke frames.
	A $\mathcal{L}_m$-sentence $\varphi$ is a {\em semantic consequence} of $\Gamma$ over $\mathcal C$, written $\Gamma\Vdash_{\mathcal C}\varphi$,
	if and only if every model of $\Gamma$ whose underlying propositional Kripke frame belongs to $\mathcal C$ forces $\varphi$.
\end{itemize}

A {\em first-order modal logic} is a set of $\mathcal L_m$-formulas containing all valid first-order formulas, the modal axiom schema $\Box(\varphi\rightarrow\psi) \rightarrow (\Box\varphi\rightarrow\Box\psi)$ and closed under the rules of Modus Ponens, Necessitation, and Universal Generalization. The smallest such logic is denoted by $\mathsf{FOK}$. More generally, if $\mathsf{L}$ is a normal propositional modal logic, then $\mathsf{FOL}$ denotes the smallest first-order modal logic extending $\mathsf{L}$.

Let $\mathsf{FOL}$ be a first-order modal logic, and let $\Gamma \cup \{\varphi\}$ be a set of $\mathcal{L}_m$-sentences.
\begin{itemize}
	\item $\varphi$ is a {\em theorem} of $\mathsf{FOL}$, written $\vdash_{\mathsf{FOL}}\varphi$, if $\varphi\in\mathsf{FOL}$.
	\item $\varphi$ is {\em derivable from $\Gamma$} in $\mathsf{FOL}$, written $\Gamma\vdash_{\mathsf{FOL}}\varphi$, if  there exist $\psi_1,\ldots,\psi_n\in \Gamma$ such that $\vdash_{\mathsf{FOL}}
	(\psi_1\land\cdots\land\psi_n)\rightarrow\varphi$.
\end{itemize}

\begin{con}
	Throughout this paper, we use the terms {\em Kripke frame} and {\em Kripke model} to mean, respectively, first-order Kripke frames and first-order Kripke models. Since our ambient first-order modal logic is $\mathsf{FOT}$, we omit the subscript $\mathsf{FOT}$ from $\vdash$ whenever no confusion can arise. Moreover, a set $\Gamma$ of $\mathcal{L}_m$-sentences is called an {\em $\mathcal{L}_m$-theory} if it is closed under deduction; that is, whenever $\Gamma\vdash\varphi$, we have $\varphi\in \Gamma$.
\end{con}
\subsection{Stone spaces and regular open algebras}

Let $\mathsf{B}$ be a  Boolean algebra. Define \[\operatorname{St}(\mathsf{B})
=\{U\subseteq\mathsf{B} : U \text{ is an ultrafilter} \}.\]
For each $b\in\mathsf{B}$, let $N_b=\{U\in\operatorname{St}(\mathsf{B}) : b\in U\}$. The Stone topology $\tau_{\mathsf{B}}$ on $\operatorname{St}(\mathsf{B})$ is the topology generated by the family $\{N_b : b\in\mathsf{B} \}$. The topological space $\langle\operatorname{St}(\mathsf{B}),\tau_{\mathsf{B}}\rangle$ is called the {\em Stone space} of $\mathsf{B}$.

\begin{lemma}\label{iso}
Assume that $T$ is a consistent recursively axiomatizable theory extending Robinson arithmetic $Q$. Then the Lindenbaum-Tarski algebras $\mathsf B_{T}$ and $\mathsf B_{T_m}$ are countable atomless Boolean algebras.
\end{lemma}

\begin{re}\label{hom}
If $\mathsf{B}$ is a countable atomless Boolean algebra, then the Stone space $\langle\operatorname{St}(\mathsf{B}),\tau_{\mathsf{B}}\rangle$ is homeomorphic to the Cantor space $2^\omega.$
\end{re}

By Stone's Representation Theorem, the map
$\phi:\mathsf{B}\to\mathsf{CLOP}(\operatorname{St}(\mathsf{B}))$,
defined by $b\mapsto N_b$, is a Boolean algebra isomorphism. Consequently, every clopen subset of $\operatorname{St}(\mathsf{B})$ is of the form $N_b$ for a unique $b\in\mathsf{B}$, and the family $\{N_b : b\in\mathsf{B}\}$ forms a basis for the Stone topology $\tau_{\mathsf{B}}$. Hence,
\[
\mathsf{B}\cong
\mathsf{CLOP}\bigl(\operatorname{St}(\mathsf{B})\bigr).
\]

Given a topological space $(X,\tau)$, an open set $U\subseteq X$ is called
{\em regular open} if $U=\mathsf{Int}(\mathsf{Cl}(U))$.
Let $\mathsf{Reg}\bigl(\operatorname{St}(\mathsf{B})\bigr)$ denote the
Boolean algebra of regular open subsets of $\operatorname{St}(\mathsf{B})$.
Since $\mathsf{CLOP}(\operatorname{St}(\mathsf{B}))$
 densely embedds into $\mathsf{Reg}\bigl(\operatorname{St}(\mathsf{B})\bigr)$, we
have
\[
\mathsf{RO}(\mathsf{B})
\cong
\mathsf{RO}\bigl(\mathsf{CLOP}(\operatorname{St}(\mathsf{B}))\bigr)
\cong
\mathsf{Reg}\bigl(\operatorname{St}(\mathsf{B})\bigr),
\]
where $\mathsf{RO}(\mathsf{B})$ denotes the Boolean completion of $\mathsf{B}$, and similarly for
$\mathsf{CLOP}(\operatorname{St}(\mathsf{B}))$.

Now suppose $\mathsf{B}$ is a countable atomless Boolean algebra.
Let $\mathcal{B}\bigl(\operatorname{St}(\mathsf{B})\bigr)$ denote the
$\sigma$-algebra of Borel subsets of $\operatorname{St}(\mathsf{B})$, and let $I$ be a proper $\sigma$-ideal on $\mathcal{B}\bigl(\operatorname{St}(\mathsf{B})\bigr)$. Then we can form the Boolean algebra $A_I=\mathcal{B}\bigl(\operatorname{St}(\mathsf{B})\bigr)/I$  and the forcing notion
$P_I={A_I}^+$.
Forcing notions of the form $P_I$ are those we are working with.
\subsection{Henkin expansions and sentence algebras}
Throughout this paper, let $T$ be a consistent recursively axiomatizable first-order theory in the countable language $\mathcal{L}$ which contains Robinson's theory $\mathsf Q$. We show that we can pass to the Henkin expansion of $\mathcal L$.

For completeness, we recall the construction. Let $\mathcal L^H$ be the language obtained from $\mathcal L$ by adding, for every formula $\exists x\,\varphi(x,\bar y)$, a new constant symbol $c_{\exists x\varphi(\bar y)}$. The theory $T^H$ is obtained by adding to $T$ the Henkin axioms 
\[
\forall \bar y\bigl(\exists x\,\varphi(x,\bar y)\rightarrow
\varphi(c_{\exists x\varphi(\bar y)},\bar y)\bigr).
\]
We call $T^H$ the Henkinization of $T$. We write $\mathcal L_m^H= \mathcal L^H\cup\{\Box,\lozenge\}$ and let $T_m^H$ denote the corresponding Henkin expansion of the modal theory $T_m$.

For a theory $S$ in a language containing only sentences, let $\mathsf B_S$ denote its Lindenbaum-Tarski sentence algebra; that is, $\mathsf B_S$ is the Boolean algebra of sentences modulo provable equivalence in $S$. Thus the elements of $\mathsf B_{T^H}$ and $\mathsf B_{T_m^H}$ are equivalence classes of sentences of $\mathcal L^H$ and $\mathcal L_m^H$, respectively. We write $[\varphi]_{T^H}$ and $[\varphi]_{T_m^H}$ for the corresponding equivalence classes.

Let $T_m$ be a consistent recursively axiomatizable normal modal extension of $T$ and work throughout with its Henkin expansion $T_m^H$.

\begin{lemma}
Let $T^H$ be the Henkin expansion of $T$. Then $T^H$ is conservative over $T$. Moreover, if $T$ is recursively axiomatizable, then so is $T^H$, and consistency of $T$ implies consistency of $T^H$.
\end{lemma}

\begin{lemma}
Assume $T^H$ is consistent. Then the Boolean algebra $\mathsf B_{T^H}$ is countable and atomless. The same holds for $\mathsf B_{T_m^H}$.
\end{lemma}

The above two lemmas allow us to go from $T$ to $T^H$ and from $T_m$ to $T_m^H$, so from now on we assume our theory $T$ is itself Henkinian.
\subsection{The Boolean-valued universe and generic extensions}

Let $\mathcal M$ be a transitive first-order model of $\mathsf{ZFC}$, and let
$\mathsf{C}\in \mathcal M$ be such that $\mathcal M$  models $\mathsf{C}$ to be a complete Boolean algebra. For each ordinal
$\alpha\in\operatorname{Ord}\cap\mathcal{M}$, define recursively
\[\mathcal M_0^{\mathsf{C}}=\varnothing,\]
\[\mathcal M_{\alpha+1}^{\mathsf{C}} = 	\mathcal{P}(\mathcal M_{\alpha}^{\mathsf{C}} \times \mathsf{C}),\]
and, for every limit ordinal $\beta$,
\[\mathcal M_{\beta}^{\mathsf{C}} = \bigcup_{\alpha<\beta} \mathcal M_{\alpha}^{\mathsf{C}}.\]
Finally, define
\[\mathcal M^{\mathsf{C}} = \bigcup_{\alpha\in\operatorname{Ord}^{\mathcal M}}
\mathcal M_{\alpha}^{\mathsf{C}}.\]
If $V$ denotes the universe of sets, then $V^{\mathsf{C}}$ is called the {\em Boolean-valued universe} over $\mathsf{C}$, and its elements are called {\em $\mathsf{C}$-names}. By absoluteness, $\mathcal M^{\mathsf{C}}=V^{\mathsf{C}} \cap \mathcal M.$
For any formula $\varphi$ in the language of set theory, let $[\![\varphi]\!]$ denote its associated $\mathsf{C}$-value.
It  is defined recursively by:
\[
[\![x \subseteq y]\!]
=
\bigwedge_{u\in \operatorname{dom}(x)}
\Bigl(
(x(u)\rightarrow [\![u\in y]\!])
\Bigr),
\]
\[
[\![x = y]\!] = [\![x \subseteq y]\!] \wedge [\![y \subseteq x]\!],
\]
\[
[\![x\in y]\!]
=
\bigvee_{u\in \operatorname{dom}(y)}
\bigl(y(u)\wedge[\![x=u]\!]\bigr),
\]
\[
[\![\neg\psi]\!]=\neg[\![\psi]\!],
\qquad
[\![\psi\wedge\chi]\!]=[\![\psi]\!]\wedge[\![\chi]\!],
\]
\[
[\![\psi\vee\chi]\!]=[\![\psi]\!]\vee[\![\chi]\!],
\qquad
[\![\psi\to\chi]\!]=[\![\psi]\!]\rightarrow[\![\chi]\!],
\]
\[
[\![\forall x\,\psi(x)]\!]
=
\bigwedge_{a\in V^{\mathsf C}}[\![\psi(a)]\!],
\qquad
[\![\exists x\,\psi(x)]\!]
=
\bigvee_{a\in V^{\mathsf C}}[\![\psi(a)]\!].
\]

\begin{definition}
	Let $\mathcal{M}$ be a transitive model of $\mathsf{ZFC}$, and let
	$\mathsf{P}\in\mathcal{M}$ be a partially ordered set. A filter
	$G\subseteq\mathsf{P}$ is said to be \emph{$\mathcal{M}$-generic} over
	$\mathsf{P}$ if $G\cap D\neq\varnothing$ for every predense subset $D\in\mathcal{M}$ of $\mathsf{P}$.
	
	Similarly, if $\mathsf{C}\in\mathcal{M}$ is a Boolean algebra, then an
	ultrafilter $G$ is said to be
	{\em $\mathcal{M}$-generic} over $\mathsf{C}$ if $G$ is
	$\mathcal{M}$-generic over the partial order $\mathsf{C}^{+}=\mathsf{C} \setminus \{0_{\mathsf{C}}\}$.
\end{definition}

Let $\mathcal{M}$ be a  transitive model of $\mathsf{ZFC}$, let 	$\mathsf{C}\in\mathcal{M}$ be a complete Boolean algebra, and let 	 $G\in\operatorname{St}(\mathsf{C})$ be an ultrafilter. Then we can form the model
\[\mathcal{M}[G] = \{\utilde{X}[G] : \utilde{X}\in \mathcal M^{\mathsf{C}}\},\]
where $\utilde{X}[G]$ is defined by induction as
\[\utilde{X}[G] = \{\utilde{Y}[G] : \exists\,p\in G\; \bigl(\langle\utilde{Y},p\rangle\in\utilde{X}\bigr)\}.\]
 $\mathcal{M}[G]$ is a  transitive model  containing $\mathcal M\cup\{G\}$.
If further $G$ is a $\mathsf{C}$-generic ultrafilter over $\mathcal{M}$, then
 $\mathcal{M}[G]$ is a  model of $\mathsf{ZFC}$.
\section{Set-theoretic translation} \label{Sec:Trans}

Fix
$\mathsf B=\mathsf B_{T},
$
and let
$
\mathcal B\bigl(\operatorname{St}(\mathsf B)\bigr)
$
denote the $\sigma$-algebra of Borel subsets of
$\operatorname{St}(\mathsf B)$. Let $I$ be a proper $\sigma$-ideal on
$\mathcal B(\operatorname{St}(\mathsf B))$ such that
$
I\cap\mathsf{CLOP}(\operatorname{St}(\mathsf B))
=
\{\emptyset\}.
$
Set
\[\mathsf A_I
=
\mathcal B\bigl(\operatorname{St}(\mathsf B)\bigr)/I.
\]
We force with the forcing notion
$
\mathsf P_I=\mathsf A_I^+,
$
and its Boolean completion
$
\mathsf C_I=\mathsf{RO}(\mathsf P_I).
$
We  assume that
$
\mathsf A_I\subseteq \mathsf C_I.
$

For every $\mathcal L$-sentence $\varphi$, put
\[
N_{[\varphi]}
=
\{U\in\operatorname{St}(\mathsf B):[\varphi]\in U\}
\]
and define
\[
[\varphi]_I
=
N_{[\varphi]}/I\in\mathsf C_I.
\]
Then by the assumption on $I$, the map
\[
\iota:\mathsf B\longrightarrow\mathsf C_I,
\qquad
\iota([\varphi])=[\varphi]_I,
\]
is an injective Boolean algebra homomorphism.

\begin{re}
	If $I=I_{\mathrm{meager}}$, then $\mathsf P_I$ is forcing
	equivalent to Cohen forcing. If, after identifying
	$\operatorname{St}(\mathsf B)$ with the Cantor space, $I$ is the
	null ideal for the standard product measure, then $\mathsf P_I$ is
	forcing equivalent to random forcing.
\end{re}

Let $\utilde{\Gamma}
	=
	\left\{
	\left\langle
	{\varphi},[\varphi]_I
	\right\rangle:
	\varphi\in\operatorname{Sent}(\mathcal L)
	\right\}.$ Thus  $\utilde{\Gamma}$ is a $\mathsf C_I$-name
	and clearly,
	\[
	\Vdash_{\mathsf C_I}
	``\utilde{\Gamma}\text{ is a complete consistent first-order
	theory extending } T".
	\]
We next construct the first-order structure which will be used in the
translation.
Let
$
\operatorname{\utilde{St}}(\mathsf C_I)
$
be a $\mathsf C_I$-name such that
\[
\Vdash_{\mathsf C_I}
``\operatorname{\utilde{St}}(\mathsf C_I)
=
\{U\subseteq {\mathsf C}_I:
U\text{ is an ultrafilter on }{\mathsf C}_I\}".
\]
Note that
$
\Vdash_{\mathsf C_I}
``\dot G_{\mathsf C_I}
\in
\operatorname{\utilde{St}}(\mathsf C_I)".
$
For notational simplicity, set
$
\utilde D
=
\operatorname{\utilde{St}}(\mathsf C_I).
$
Let  $\utilde{J}$ be a $\mathsf{C}_I$-name satisfying
	\[	\Vdash_{\mathsf{C}_I} ``\utilde{J}: C(\mathcal{L}) \longrightarrow
	\operatorname{\utilde{St}}(\mathsf{C}_I)",\]
where $C(\mathcal{L})$ is the set of constant symbols of $\mathcal{L}$.
\begin{lemma}\label{lem:stone-model}
Fix $\utilde{J}$ as above.	There is a canonical $\mathsf C_I$-name
 \[
\utilde{\mathcal{M}}_J
=
\left\langle
\utilde D,
(\utilde f)_{f\in\mathcal L},
(\utilde R)_{R\in\mathcal L},
(\utilde J(c))_{c\in C(\mathcal L)}
\right\rangle.
\]
such that
$
	\Vdash_{\mathsf C_I}
	``\utilde{\mathcal{M}}_J
	\text{ is an }\mathcal L\text{-structure with domain }
	\utilde D
	\text{ and }
	\utilde{\mathcal M}\models\utilde{\Gamma}  \text{''}.
$
\end{lemma}

\begin{proof}
We define the interpretation of the symbols of $
\mathcal L$ inside the Boolean-valued universe.  Since
$\utilde D=\operatorname{\utilde{St}}(\mathsf C_I)$ is forced to be the
set of ultrafilters on $\mathsf C_I$, it remains to define the
interpretations of functions, relations, and constants.

For every $n$-ary function symbol $f\in\mathcal L$, define
$\utilde f$ by the requirement that
\[
\Vdash_{\mathsf C_I}
``\utilde f(\utilde u_0,\ldots,\utilde u_{n-1})=\utilde v
\quad\Longleftrightarrow\quad
[\![f(\utilde u_0,\ldots,\utilde u_{n-1})=\utilde v]\!]\in
\dot G_{\mathsf C_I}  \text{''}.
\]

Similarly, for every $n$-ary relation symbol $R\in\mathcal L$, put
\[
\Vdash_{\mathsf C_I}
``R(\utilde u_0,\ldots,\utilde u_{n-1})
\quad\Longleftrightarrow\quad
[\![R(\utilde u_0,\ldots,\utilde u_{n-1})]\!]
\in\dot G_{\mathsf C_I} \text{''}.
\]
For a constant symbol $c\in C(\mathcal L)$ its interpretation is
$\utilde J(c)$ by definition.

We can now easily see, by induction on sentences of $\mathcal L$, that
\[
\Vdash_{\mathsf C_I}
``\utilde{\mathcal M}_J\models\varphi
\quad\Longleftrightarrow\quad
[\varphi]_I\in\dot G_{\mathsf C_I}  \text{''}.
\]
Finally note that as every axiom of $T$ has Boolean value one, every axiom of
$\utilde\Gamma$ is forced in $\utilde{\mathcal M}_J$. Therefore
$
\Vdash_{\mathsf C_I}
``\utilde{\mathcal M}_J\models\utilde\Gamma \text{''}.
$
This proves the lemma.
\end{proof}

Fix, from now on, $\utilde{\mathcal M}_J$ as above.
Let $\utilde N$ be a $\mathsf C_I$-name such that
\[
\Vdash_{\mathsf C_I}
``\utilde N\subseteq\dot G_{\mathsf C_I}
\text{ is a filter on }{\mathsf C}_I".
\]

\begin{definition}\label{def:translation0}
	For every $\mathcal L_m$-formula $\varphi(\vec x)$ and every
	appropriate tuple $\utilde{\vec a}$, we recursively define an element
	\[
	b(\varphi,\utilde{\vec a})=b_{\utilde N, \utilde J}(\varphi,\utilde{\vec a})
	\in\mathsf C_I
	\]
as follows:
\begin{itemize}
\item
	$
	b
	(t_1=t_2,\utilde{\vec a})
	=
	\left[\!\left[
	t_1^{\utilde{\mathcal M}_J}(\utilde{\vec a})
	=
	t_2^{\utilde{\mathcal M}_J}(\utilde{\vec a})
	\right]\!\right],
	$
\item
	$b
	(R(t_1,\ldots,t_k),\utilde{\vec a})
	=
	\left[\!\left[
	\utilde R
	\bigl(
	t_1^{\utilde{\mathcal M}_J}(\utilde{\vec a}),
	\ldots,
	t_k^{\utilde{\mathcal M}_J}(\utilde{\vec a})
	\bigr)
	\right]\!\right].
$

\item $
	b(\neg\varphi,\utilde{\vec a})
	=
	\neg b(\varphi,\utilde{\vec a}),
	$
\item
	$ b(\varphi\wedge\psi,\utilde{\vec a})
	=
	b(\varphi,\utilde{\vec a})
	\wedge
	b(\psi,\utilde{\vec a}).
$

\item $
	b
	(\forall x\,\varphi,\utilde{\vec a})
	=
	\bigwedge_{\utilde u\in\operatorname{dom}(\utilde D)}
	\left(
	\utilde D(\utilde u)
	\to
	b
	(\varphi,\utilde{\vec a}[x\mapsto\utilde u])
	\right).
$
\item $
	b
	(\Box\varphi,\utilde{\vec a})
	=
	\left[\!\left[
	{
	b(\varphi,\utilde{\vec a})
	}
	\in\utilde N
	\right]\!\right].
$
\end{itemize}
\end{definition}

We are now ready to define the set theoretic translation.
\begin{definition}\label{def:translation}
Suppose $\varphi(\vec x)$ is an $\mathcal L_m$-formula
	 $\utilde{\vec a}$ is a $\mathsf C_I$ for elements of $\utilde D$.	The set-theoretic translation
$
\nu_{\utilde N,\utilde J}
	(\varphi)(\utilde{\vec a})$
	is defined as the $\mathcal L_{\in}$-sentence
	\[
\nu_{\utilde N,\utilde J}
	(\varphi)(\utilde{\vec a}):=	{
	b_{\utilde N, \utilde J}(\varphi,\utilde{\vec a})
	}
	\in
	\dot G_{\mathsf C_I}.
	\]
When there is no confusion, we delete the subscripts $\utilde N,\utilde J$ and  write $
	\utilde{\nu}(\varphi)(\utilde{\vec a})
	=
	\nu_{\utilde N,\utilde J}
	(\varphi)(\utilde{\vec a}).
$
\end{definition}

\begin{lemma}\label{lem:first-order-value}
	For every $\mathcal L$-sentence $\varphi$, we have
$
	b_{\utilde N,\utilde J}(\varphi)
	=
	[\varphi]_I.
$
	Consequently,
\[\left[\!\left[
\nu_{\utilde N,\utilde J}(\varphi)
\right]\!\right]
=
[\varphi]_I.\]
\end{lemma}

\begin{proof}
It is easily seen that
\[
[\varphi]_I = \left[\!\left[
	{\varphi}\in\utilde{\Gamma}
	\right]\!\right]
	\leq
	\left[\!\left[
	\utilde{\mathcal M}_{J}\models{\varphi}
	\right]\!\right]
	=
	b_{\utilde N, \utilde J}(\varphi).
\]	
Applying the same argument to $\neg\varphi$, we obtain
	\[
	[\neg\varphi]_I
	\leq
	b_{\utilde N, \utilde J}(\neg\varphi)
	=
	\neg b_{\utilde N, \utilde J}(\varphi).
	\]
It follows that
$
	b_{\utilde N, \utilde J}(\varphi)
	=
	[\varphi]_I.
$
	The second assertion follows immediately.
\end{proof}

\begin{coro}\label{cor:first-order-forcing}
	If $\varphi$ is an $\mathcal L$-sentence and $T\vdash\varphi$, then
$
	\Vdash_{\mathsf C_I}
	\nu_{\utilde N,\utilde J}(\varphi).
$
\end{coro}

\begin{proof}
	If $T\vdash\varphi$, then
$
	[\varphi]=1_{\mathsf B_T},$
	and therefore
$
	[\varphi]_I=1_{\mathsf C_I}.
$
	The result follows from Lemma~\ref{lem:first-order-value}.
\end{proof}
\section{Soundness and completeness} \label{Sec:Com}

Throughout this section, we adopt the notation introduced in the previous section. In particular, we write $\utilde{\nu}(\varphi(\vec{x})) := \nu_{\utilde{N},\utilde{J}}(\varphi(\vec{x}))$, and define $\utilde{\nu}(T):= \{\utilde{\nu}(\varphi): \varphi\in T\}$ and
$\utilde{\nu}(T_m):= \{\utilde{\nu}(\varphi) : \varphi\in T_m\}$.

\begin{lemma}
	For every $\varphi\in\mathcal{L}$, if
	$T\vdash\varphi$, then $\Vdash_{\mathsf{C}_I} 	``\utilde{\nu}(T)\vdash\utilde{\nu}(\varphi)$".
\end{lemma}
\begin{proof}
	We argue by induction on the length of a derivation of $\varphi$ from
	$T$.
	
	If $\varphi\in T$, then
	$\utilde{\nu}(\varphi)\in\utilde{\nu}(T)$ by definition, and hence $\Vdash_{\mathsf{C}_I}
	``\utilde{\nu}(T)\vdash\utilde{\nu}(\varphi)$".
	
	Next, suppose that $\varphi$ is an instance of a first-order axiom.
	Since $T\vdash\varphi$, we have $[\varphi] = 1_{\mathsf{B}}$ and therefore $N_{[\varphi]} =  \operatorname{St}(\mathsf{B})$. Consequently, $[\varphi]_I = 1_{\mathsf{P}_I}$ and hence $\Vdash_{\mathsf{C}_I} ``[\varphi]_I\in\dot{G}_{\mathsf{C}_I}"$. This implies that $\Vdash_{\mathsf{C}_I}
	\utilde{\nu}(\varphi)$.
	
	Now, assume that $\varphi$ is obtained from $\psi$ and $\psi\rightarrow\varphi$ by the MP rule. By the induction hypothesis, $\Vdash_{\mathsf{C}_I} ``\utilde{\nu}(T)\vdash\utilde{\nu}(\psi)"$ and $\Vdash_{\mathsf{C}_I} ``\utilde{\nu}(T)\vdash \utilde{\nu}(\psi\rightarrow\varphi)"$.
	Since \[[\![\psi\wedge(\psi\rightarrow\varphi)]\!] \leq [\![\varphi]\!],\]
	it follows that $\Vdash_{\mathsf{C}_I} 	``\utilde{\nu}(T)\vdash\utilde{\nu}(\varphi)"$.

	Finally, suppose that $\varphi$ is of the form $\forall x\,\psi$ and is obtained from $\psi$ by the Generalization rule. By the induction hypothesis, $\Vdash_{\mathsf{C}_I} ``\utilde{\nu}(T)\vdash \utilde{\nu}(\psi)"$. By the recursive definition of the translation,
	\[\utilde{\nu}(\forall x\,\psi) = \forall x \bigl( x\in\operatorname{\utilde{St}}(\mathsf{P}_I) \rightarrow \utilde{\nu}(\psi) \bigr). \]
	Hence, $\Vdash_{\mathsf{C}_I} ``\utilde{\nu}(T)\vdash \utilde{\nu}(\forall x\,\psi)"$. This completes the induction.
\end{proof}

\begin{lemma}
	$\Vdash_{\mathsf{C}_I} \utilde{\nu}(\chi)$ holds for $\chi\in \mathcal{L}_m$ where $\chi$ is one of the following forms
	\begin{enumerate}
		\item $\Box \varphi \to \varphi$;
		\item $\Box (\varphi\to \psi)  \to (\Box\varphi\to \Box \psi)$.
	\end{enumerate}
\end{lemma}
\begin{proof}
	\begin{enumerate}
		\item It suffices to show that $[\![\Box\varphi\rightarrow\varphi]\!]=1$.
		Equivalently, \[[\![[\![\varphi]\!]\in\utilde N]\!] 	\le [\![\varphi]\!].\]
		Since $[\![\varphi]\!] = [\![[\![\varphi]\!]\in\dot G_{\mathsf C_I}]\!]$
		and $\Vdash_{\mathsf C_I} ``\utilde N\subseteq\dot G_{\mathsf C_I}"$, 	we obtain
		\[	[\![[\![\varphi]\!]\in\utilde N]\!] \le [\![\varphi]\!].\]
		\item It is enough to prove that $[\![\Box(\varphi\rightarrow\psi)\wedge\Box\varphi]\!] 	\le
		[\![\Box\psi]\!]$.
		Indeed,
		\begin{align*}
			[\![\Box(\varphi\rightarrow\psi)\wedge\Box\varphi]\!]
			&=
			[\![
			[\![\varphi\rightarrow\psi]\!]\in\utilde N
			\wedge
			[\![\varphi]\!]\in\utilde N
			]\!]\\
			&\le
			[\![
			([\![\varphi\rightarrow\psi]\!]\wedge[\![\varphi]\!])
			\in\utilde N
			]\!],
		\end{align*}
		since $\Vdash_{\mathsf C_I} ``\utilde N \text{ is a filter}"$.
		Moreover, $[\![\varphi\rightarrow\psi]\!] \wedge [\![\varphi]\!] \le 	[\![\psi]\!]$, and therefore, by the monotonicity of filters,
		\[ [\![ ([\![\varphi\rightarrow\psi]\!]\wedge[\![\varphi]\!]) \in\utilde N
		]\!] \le [\![ [\![\psi]\!] 	\in\utilde N ]\!].	\]
		Hence,
		\[[\![\Box(\varphi\rightarrow\psi)\wedge\Box\varphi]\!] \le [\![ [\![\psi]\!]
		\in\utilde N ]\!] = [\![\Box\psi]\!], \]
		which completes the proof.
	\end{enumerate}
\end{proof}

\begin{lemma}\label{lem:necessitation}
	If $\Vdash_{\mathsf{C}_I} ``\utilde{\nu}(T_m)\vdash
	\utilde{\nu}(\varphi)"$, then
	$\Vdash_{\mathsf{C}_I} ``\utilde{\nu}(T_m)\vdash
	\utilde{\nu}(\Box\varphi)"$.
\end{lemma}
\begin{proof}
	Assume that
	\[
	\Vdash_{\mathsf{C}_I}
	``\utilde{\nu}(T_m)\vdash\utilde{\nu}(\varphi)" .
	\]
	Hence there is a finite set
	$F\subseteq T_m$ such that
	\[
	\left[\!\left[
	\left(\bigwedge_{\psi\in F}\utilde{\nu}(\psi)\right)
	\rightarrow\utilde{\nu}(\varphi)
	\right]\!\right]=1 .
	\]
	Hence
	\[
	\left[\!\left[
	\bigwedge_{\psi\in F}\utilde{\nu}(\psi)
	\right]\!\right]
	\leq
	\left[\!\left[
	\utilde{\nu}(\varphi)
	\right]\!\right].
	\]

	For each $\psi\in F$, since $\psi\in T_m$, we have
	\[
	\Vdash_{\mathsf{C}_I}
	``\left[\!\left[\utilde{\nu}(\psi)\right]\!\right]
	\in\utilde N".
	\]
	Because $\utilde N$ is forced to be a filter, it is closed under
	finite meets. Therefore,
	\[
	\Vdash_{\mathsf{C}_I}
	``\left[\!\left[
	\bigwedge_{\psi\in F}\utilde{\nu}(\psi)
	\right]\!\right]\in\utilde N".
	\]
	Using the upward closure of filters and the inequality above, we get
	\[
	\Vdash_{\mathsf{C}_I}
	``\left[\!\left[
	\utilde{\nu}(\varphi)
	\right]\!\right]\in\utilde N".
	\]

	By the definition of the translation of the modal operator in the
	Henkin expansion,
	\[
	\utilde{\nu}(\Box\varphi)=
	\left[\!\left[\utilde{\nu}(\varphi)\right]\!\right]\in\utilde N .
	\]
	Hence
$
	\Vdash_{\mathsf{C}_I}
	``\utilde{\nu}(T_m)\vdash
	\utilde{\nu}(\Box\varphi) \text{''} ,
$
	which proves the lemma.
\end{proof}

As an immediate consequence of the preceding lemmas, we obtain the following soundness theorem.

\begin{theo}[Soundness]\label{theo:soundness}
	Suppose $T_m\vdash\varphi$. Then for all $\utilde{N}$ and $\utilde{J}$ as above,
	\[\Vdash_{\mathsf{C}_I}``\nu_{\utilde{N}, \utilde{J}}(T_m) \vdash \nu_{\utilde{N}, \utilde{J}}(\varphi)".\]
\end{theo}

Before considering the completeness problem,  we prove the following lemma.
\begin{lemma}
\label{lem:henkin-transport}
Let $\pi:\mathsf B_m\to\mathsf B$ be a Boolean algebra
isomorphism and let
\[
h:\operatorname{St}(\mathsf B)\longrightarrow
\operatorname{St}(\mathsf B_m)
\]
be the induced Stone homeomorphism. Suppose that $I_m$ is the
transport of $I$ by $h$. Then there is an induced isomorphism
\[
\widehat h:\mathsf C_I\longrightarrow\mathsf C_{m,I_m}
\]
between the corresponding Boolean completions. Moreover, if
$\utilde N$ and $\utilde J$ are $\mathsf C_I$-names, then there are
canonically associated $\mathsf C_{m,I_m}$-names
$\widehat h(\utilde N)$ and $\widehat h(\utilde J)$ such that for every
sentence $\theta\in\mathcal L_m$,
\[
\Vdash_{\mathsf C_I}
``\nu_{\utilde N,\utilde J}(\theta)"
\]
if and only if
\[
\Vdash_{\mathsf C_{m,I_m}}
``\nu_{\widehat h(\utilde N),\widehat h(\utilde J)}
(\theta)".
\]
\end{lemma}

\begin{proof}
The homeomorphism $h$ induces an isomorphism of Borel algebras and,
by the definition of $I_m$, an isomorphism of the quotient Boolean
algebras. Passing to regular open completions gives
$\widehat h$. Boolean-valued names are transported recursively by
replacing every Boolean coefficient $b$ appearing in a name by
$\widehat h(b)$.

The assertion is proved by induction on the complexity of $\theta$.
The atomic case follows from the definition.
The Boolean connectives are preserved because $\widehat h$ is a
Boolean algebra isomorphism. The quantifier step follows also immediately.
Finally,
$
\nu(\Box\theta)=[[\nu(\theta)]]\in\utilde N
$
is preserved because both Boolean values and the filter name are
transported by the same Boolean algebra isomorphism.
\end{proof}

We now prove the converse, namely the completeness theorem.

\begin{theo}[Completeness] \label{theo:completeness}
	Assume $\varphi\in \mathcal{L}_m$ and for all $\utilde{N}$ and
	$\utilde{J}$,
	\[
	\Vdash_{\mathsf{C}_I}``\nu_{\utilde{N}, \utilde{J}}(T_m)
	\vdash \nu_{\utilde{N}, \utilde{J}}(\varphi)".
	\]
	Then $T_m\vdash\varphi$.
\end{theo}

\begin{proof}
	We prove the contrapositive. Assume that
	$T_m\nvdash\varphi$. By Lemma \ref{iso} and the Cantor--Tarski theorem, there is a Boolean
	algebra isomorphism
	\[
	\pi:\mathsf B_m\longrightarrow\mathsf B .
	\]

	Since $T_m\nvdash\varphi$, we have
$
	[\neg\varphi]_m\neq 0_{\mathsf B_m}.
$
	Therefore
$
	\pi([\neg\varphi]_m)=[(\neg\varphi)^\pi]\neq0_{\mathsf B}.
$
	Let
$
	a=\pi([\neg\varphi]_m).
$

	The Stone dual of $\pi$ is a homeomorphism
	\[
	h:\operatorname{St}(\mathsf B)\longrightarrow
	\operatorname{St}(\mathsf B_m)
	\]
	given by
$
	h(U)=\pi^{-1}[U].
$
	Let $I_m$ be the
	transport of $I$ through $h$, so 
\[X\in I_m\iff h[X]\in I .\]
	Then $h$ induces an isomorphism of quotient forcing notions
	\[
	\bar h:
	\mathcal B(\operatorname{St}(\mathsf B))/I
	\longrightarrow
	\mathcal B(\operatorname{St}(\mathsf B_m))/I_m .
	\]
	Consequently the regular open completions are isomorphic. Denote the
	induced Boolean algebra isomorphism by
	\[
	\widehat h:\mathsf C_I\longrightarrow\mathsf C_{m,I_m}.
	\]

 Since $a>0$, the corresponding quotient element
$
	a_I=N[(\neg\varphi)^\pi]/I
$
	is non-zero. Hence we may choose a $\mathsf C_I$-name
	$\utilde U^*$ such that
	\[
	\Vdash_{\mathsf C_I}
	``a_I\in\utilde U^*\in
	\utilde{\operatorname{St}}(\mathsf C_I)".
	\]
	Define
$
	\Vdash_{\mathsf C_I}
	``\utilde N^*=
	\utilde U^*\cap\dot G_{\mathsf C_I}".
$
	By Lemma \ref{lem:first-order-value}, 
	\[
	\Vdash_{\mathsf C_I}
	\left[
	\left[
	\nu_{\utilde N^*,\utilde J}
	((\neg\varphi)^\pi)
	\right]
	\right]
	=
	a_I .
	\]
	Hence
$
	\Vdash_{\mathsf C_I}
	``\nu_{\utilde N^*,\utilde J}
	((\neg\varphi)^\pi)",
$
and in particular,
	\[
	\Vdash_{\mathsf C_I}
	``\nu_{\utilde N^*,\utilde J}(T)
	\nvdash
	\nu_{\utilde N^*,\utilde J}(\varphi^\pi)".
	\]
Let
$
	\widehat\pi:\mathsf C_{m,I_m}\longrightarrow\mathsf C_I
$
	be the induced isomorphism of Boolean completions, then the inverse
	image of $\utilde N^*$ and $\utilde J$ gives $\mathsf C_{m,I_m}$-names
	$\utilde N,\utilde J$ satisfying
	\[
	\Vdash_{\mathsf C_{m,I_m}}
	``\nu_{\utilde N,\utilde J}(T_m)
	\nvdash
	\nu_{\utilde N,\utilde J}(\varphi)".
	\]
	Thus there exist names for which the translated consequence fails.
	This contradicts the assumption of the theorem.
\end{proof}

Note that we indeed proved the following stronger result.

\begin{theo}\label{com}
	Suppose $\varphi\in \mathcal{L}_m$ and $T_m \nvdash \varphi$. 	
	Then there is a $\mathsf{C}$-name $\utilde{\mathcal{U}}$ such that
	\begin{enumerate}
		\item $\Vdash_{\mathsf{C}_I}$ `` $\utilde{\mathcal{\dot{U}}}$ is an ultrafilter on $\mathsf{C}_I$";
		\item If $\utilde{N}$ is such that
		$\Vdash_{\mathsf{C}_I} ``\utilde{N}= \utilde{\mathcal{U}}\cap \dot{G}_{\mathsf{C}_I}"$, then we have
	\[\Vdash_{\mathsf{C}_I}  ``\nu_{\utilde{N}, J}(T_m)  \nvdash \nu_{\utilde{N}, J}(\varphi)".\]
\end{enumerate}
\end{theo}
\section{Kripke frames and models} \label{Sec:Kripke}

In this section we associate Kripke frames with the forcing
interpretation developed above. The main point is to show that the
Boolean-valued interpretation of the modal operator agrees with the
usual Kripke interpretation over the accessibility relation induced by
the filter names.

\begin{definition}[Kripke frames and models]
Let $\mathbb L$ be the collection of all $\mathsf C_I$-names $\utilde N$
such that for some $\mathsf C_I$-name $\utilde{\mathcal U}$ as in
Theorem~\ref{com},
\[
\Vdash_{\mathsf C_I}
``\utilde N=\utilde{\mathcal U}\cap\dot G_{\mathsf C_I}".
\]
For $\utilde N\in\mathbb L$ and a name $\utilde J$, define
\[
\mathfrak M(\utilde N,\utilde J)
=
\langle\mathcal W(\utilde N),
\triangleleft_{\utilde N},
\{\mathcal M_G:G\in\mathcal W(\utilde N)\}\rangle .
\]

The set of worlds is
\[
\mathcal W(\utilde N)=
\{V[G]:G\text{ is a }\mathsf C_I\text{-ultrafilter}\}.
\]

For $G,H\in\mathcal W(\utilde N)$ we put
\[
V[G]\triangleleft_{\utilde N}V[H]
\]
iff for every modal formula $\varphi(\vec x)$ and every tuple
of names $\utilde{\vec a}$,
\[
V[G]\models
\nu_{\utilde N,\utilde J}(\varphi)
(\utilde{\vec a}[G])
\Longrightarrow
V[H]\models
\nu_{\utilde N,\utilde J}(\varphi)
(\utilde{\vec a}[H]).
\]

For each ultrafilter $G$, let
\[
\mathcal M_G=
\langle
\operatorname{\utilde{St}}(P_I)[G],
(f_G)_{f\in\mathcal L},
(R_G)_{R\in\mathcal L},
(c_G)_{c\in\mathcal L}
\rangle .
\]
be defined as in Lemma \ref{lem:stone-model}.
The associated Kripke frame is
\[
\mathfrak F(\utilde N,\utilde J)
=
\langle
\mathcal W(\utilde N),
\triangleleft_{\utilde N},
D
\rangle ,
\]
where
\[
D=
\{\operatorname{\utilde{St}}(P_I)[G]:G\in\mathcal W(\utilde N)\}.
\]
\end{definition}

\begin{definition}
For every $\utilde N\in\mathbb L$, define
\[
V[G]\Vdash_{\utilde N}\varphi(\utilde{\vec a})
\iff
V[G]\models
\nu_{\utilde N,\utilde J}(\varphi)
(\utilde{\vec a}[G]).
\]
\end{definition}

\begin{lemma}\label{lem:section5-induction}
For every formula $\varphi$ of $\mathcal L_m$,
\[
V[G]\Vdash_{\utilde N}\varphi
\iff
\mathfrak M(\utilde N,\utilde J),V[G]\models\varphi .
\]
\end{lemma}

\begin{proof}
We argue by induction on the complexity of $\varphi$.
The only non-trivial case is the modal case. Suppose first that
$
V[G]\Vdash_{\utilde N}\Box\psi .
$
By the definition of the translation,
$
[\![\nu_{\utilde N,\utilde J}(\psi)]\!]\in
\utilde N[G].
$
Let
$
V[G]\triangleleft_{\utilde N}V[H].
$
Since $\utilde N[G]\subseteq H$, we obtain
$
[\![\nu_{\utilde N,\utilde J}(\psi)]\!]\in H,
$
and therefore
\[
V[H]\Vdash_{\utilde N}\psi .
\]

Conversely, assume every accessible world satisfies $\psi$. If
$
[\![\nu_{\utilde N,\utilde J}(\psi)]\!]\notin\utilde N[G],
$
extend $\utilde N[G]$ to the ultrafilter
$\utilde{\mathcal U}[G]$ witnessing the definition of $\utilde N$.
The corresponding world is accessible from $V[G]$, but the Boolean
value of $\neg\psi$ belongs to this ultrafilter, contradiction.
Hence
$
[\![\nu_{\utilde N,\utilde J}(\psi)]\!]\in\utilde N[G],
$
and therefore
\[
V[G]\Vdash_{\utilde N}\Box\psi .
\]
\end{proof}

\begin{theo}[Completeness for Kripke models]
For every $\varphi\in\mathcal L_m$, the following are equivalent:
\begin{enumerate}
\item $T_m\vdash\varphi$,

\item
$\mathfrak F(\utilde N,\utilde J)\models\varphi,
$
for every admissible pair of names $\utilde N,\utilde J$.
\end{enumerate}
\end{theo}

\begin{proof}
Assume first that $T_m\vdash\varphi$. By Theorem~\ref{theo:completeness},
$
\Vdash_{\mathsf C_I}
``\nu(T_m)\vdash\nu(\varphi)".
$
Hence every ultrafilter evaluation satisfies the translation of
$\varphi$. By Lemma~\ref{lem:section5-induction}, every associated
Kripke model satisfies $\varphi$.

Conversely, assume that $T_m\nvdash\varphi$. By Theorem~\ref{com}
there are  names $\utilde N,\utilde J$ such that
$
\Vdash_{\mathsf C_I}
``\nu(T_m)\nvdash\nu(\varphi)".
$
Therefore some Boolean-valued interpretation satisfies the translated
theory together with the negation of the translation of $\varphi$.
By Lemma~\ref{lem:section5-induction}, the associated Kripke model
satisfies $T_m$ but fails $\varphi$.
\end{proof}


\end{document}